\documentclass[12pt]{article}

\usepackage{amsmath,amsthm}
\usepackage{amssymb}
\usepackage{color}
\usepackage{breqn}
\usepackage{enumerate}
\usepackage{graphicx}
\usepackage{bm}
\usepackage{bbm}
\usepackage[affil-it]{authblk}
\usepackage{tabu}
\usepackage{bold-extra}
\usepackage[hang,flushmargin]{footmisc} 
\usepackage[hypertex]{hyperref}
\usepackage{mathtools}
 \usepackage{relsize}
 
\newtheorem{theorem}{Theorem}

\newtheorem{lemma}[theorem]{Lemma}
\newtheorem{proposition}[theorem]{Proposition}

\theoremstyle{definition}

\newcounter{boldSectionCounter}
\stepcounter{boldSectionCounter}
\newcounter{boldSubsectionCounter}
\stepcounter{boldSubsectionCounter}

\newcommand{\boldSubsection}[1]{
	 \addtocounter{boldSectionCounter}{-1}
   \noindent {\bfseries{\scshape \arabic{boldSectionCounter}.\arabic{boldSubsectionCounter}. #1}}\\[6pt]
   \stepcounter{boldSubsectionCounter}
   \addtocounter{boldSectionCounter}{1}
}		

\newcommand{\boldSection}[1]{
   \large\begin{center}\noindent {\bfseries{\scshape \S \arabic{boldSectionCounter} #1}}\\[6pt]\end{center}\normalsize
   \stepcounter{boldSectionCounter}
	 \setcounter{boldSubsectionCounter}{1}
}	
\newcommand\restr[2]{{
  \left.\kern-\nulldelimiterspace 
  #1 
  \vphantom{\big|} 
  \right|_{#2} 
}}	

\newcommand\blfootnote[1]{%
  \begingroup
  \renewcommand\thefootnote{}\footnote{#1}%
  \addtocounter{footnote}{-1}%
  \endgroup
}

\makeatletter
\newcommand*\bcdot{\mathpalette\bigcdot@{0.5}}
\newcommand*\bigcdot@[2]{\mathbin{\vcenter{\hbox{\scalebox{#2}{$\m@th#1\bullet$}}}}}
\makeatother

\makeatletter
\def\blfootnote{\gdef\@thefnmark{}\@footnotetext}
\makeatother	
	
\begin{document}
\begin{center}\Large\noindent{\bfseries{\scshape An improved upper bound for the grid Ramsey problem}}\\[12pt]
\normalsize\noindent{\scshape Luka Mili\'cevi\'c}\\[24pt]\end{center}

\footnotesize
\begin{center}\sc{\textbf{Abstract}}\end{center}

\begin{changemargin}{1in}{1in}
\hspace{12pt}~For a positive integer $r$, let $G(r)$ be the smallest $N$ such that, whenever the edges of the Cartesian product $K_N \times K_N$ are $r$-coloured, then there is a rectangle in which both pairs of opposite edges receive the same colour. In this paper, we improve the upper bounds on $G(r)$ by proving $G(r) \leq \Big(1 - \frac{1}{128}r^{-2}\Big) r^{\binom{r+1}{2}}$, for $r$ large enough. Unlike the previous improvements, which were based on bounds for the size of set systems with restricted intersection sizes, our proof is a form of a quasirandomness argument. 
\end{changemargin} 

\normalsize

\boldSection{Introduction}
\hspace{12pt}~This paper is concerned with the following question, known as the grid Ramsey Problem. For a positive integer $r$, let $G(r)$ be the smallest $N$ such that, whenever an $r$-edge-colouring $\chi$ of the Cartesian product $K_N \times K_N$ is given, then there is a rectangle in which both pairs of opposite edges receive the same colour, i.e. we may find $(i, j, i', j')$ such that $\chi\Big((i,j)(i',j)\Big) = \chi\Big((i,j')(i',j')\Big)$ and $\chi\Big((i,j)(i,j')\Big) = \chi\Big((i',j)(i',j')\Big)$.\footnote{The notation $\chi\Big((i,j)(i',j)\Big)$ may seem unusual at first, but the vertices of the graph $K_N \times K_N$ are ordered pairs and if we follow the standard convention of writing $uv$ for edge between vertices $u$ and $v$ in a graph, then we end up with an expression like $\chi\Big((i,j)(i',j)\Big)$.} The grid Ramsey problem is to determine $G(r)$. In addition to being interesting in its own right, this function is related to other topics in Ramsey theory. In fact, the trivial bound of $G(r) \leq r^{\binom{r+1}{2}} + 1$ is the simplest form of a crucial ingredient of Shelah's celebrated proof of primitive recursive bounds in Hales-Jewett theorem~\cite{Shelah}. Motivated by this, Graham, Rothschild and Spencer~\cite{GRS} raised the question of whether $G(r)$ can be bounded by a polynomial in $r$. In~\cite{gridRamsey}, Conlon, Fox, Lee and Sudakov answered their question in negative, by constructing an $r$-edge-colouring of $K_N \times K_N$ without alternating rectangles, where $N \geq 2^{\Omega\big((\log r)^{5/2}/\sqrt{\log \log r}\big)}$. Moreover, they constructed (Theorem 1.3 (iii) of~\cite{gridRamsey}) an $r$-edge-colouring of $K_M \times K_N$ without alternating rectangles, where $M$ was quadratic in $r$, and $N$ was close to $r^{\binom{r+1}{2}}$ (in fact $N \geq r^{\big(1 - (\log r)^{-1}\big)r^2/2}$), which to some extend indicates the difficulty of improving the upper bounds on $G(r)$.\\
\indent When it comes to upper bounds, the Shelah's original bound is $G(r) \leq r^{\binom{r+1}{2}} + 1$. That argument actually uses very little of the structure present in the problem, and we include it here since it is very short and it is our starting point in the proof. Namely, we simply pick $r+1$ favourite rows $j_1, \dots, j_{r+1}$ and for each column $i$, restrict $\chi$ to the complete graph on the vertices $(i, j_1), \dots, (i,j_{r+1})$. Since there are $r^{\binom{r+1}{2}}$ $r$-edge-colourings of $K_{r+1}$, by pigeonhole principle, there are two columns $i$ and $i'$ whose colourings of the fixed $K_{r+1}$ agree. Finally, consider $r+1$ horizontal edges $(i,j_l)(i',j_l)$ for $l \in [r+1]$. There are two that have the same colour, and determine the desired rectangle.\\
\indent The only improvements so far are due to Gy\'arf\'as~\cite{Gyarfas} who showed that
\[G(r) \leq r^{\binom{r+1}{2}} - r^{\binom{r-1}{2}} + 1 = \Big(1 - r^{-(2r - 1)}(1 - o(1))\Big)r^{\binom{r+1}{2}}\]
and a very recent one due to Corsten~\cite{Corsten}, 
\[G(r) \leq r^{\binom{r+1}{2}} - \Big(\frac{1}{4} - o(1)\Big)r^{\binom{r}{2}} = \Big(1 - \frac{1}{4}r^{-r}(1 - o(1))\Big)r^{\binom{r+1}{2}}.\]
Our main result is
\begin{theorem}\label{mainThm}For large enough $r$, we have 
\[G(r) \leq \Big(1 - \frac{1}{128}r^{-2}(1  - o(1))\Big)r^{\binom{r+1}{2}}.\]
\end{theorem}
Unlike the arguments of Gy\'arf\'as and Corsten, who used Fisher's and Ray-Chaudhuri-Wilson theorems that bound the size of a set system with restrictions on intersections, our proof is based on quasirandomness. Informally speaking, a graph is quasirandom if it contains roughly the same number of cycles of length 4 as a random graph of the same density. The reason why such graphs are called quasirandom lies in fact that this requirement forces a graph to behave like a random one, e.g. to have roughly the same number of triangles as a random graph of the same density would, etc. This notion was introduced by Thomason~\cite{ThomasonQR}, and by Chung, Graham and Wilson~\cite{CGW}. We note that quasirandomness has been playing a role in Ramsey theory as well. For example, the best known bounds on the diagonal Ramsey numbers, due to Conlon~\cite{ConlonRamsey}, are proved using such ideas, and Conlon's proof is in fact based on an earlier argument of Thomason~\cite{ThomasonRamsey}, which also exploits quasirandomness.\\

\noindent\textbf{Acknowledgements.} I acknowledge the support of the Ministry of Education, Science and Technological Development of the Republic of Serbia, Grant III044006.
\boldSection{Proof of Theorem~\ref{mainThm}}
\boldSubsection{Overview of the argument}
Suppose that $N$ is very slightly smaller $r^{\binom{r+1}{2}}$ and suppose that $\chi$ is a given $r$-edge-colouring of $K_N \times K_N$ without alternating rectangles. The Shelah's argument above tells us that every two rows, viewed as $r$-edge-colourings of $K_N$ must not coincide on $K_{r+1}$. However, since $N$ is very close to $r^{\binom{r+1}{2}}$ that means that, since we do not have $K_{r+1}$ which receives the same colours in two rows, then for any smaller graph $H$ with its own $r$-edge-colouring $c$, we must get roughly $r^{-e(H)}$ rows in which $H$ receives exactly the colouring $c$. In particular, if we define the intersection of two rows to be the graph with the edges set consisting of edges that are of same colour in both rows, then the intersection graphs are quasirandom. On the other hand, looking at the vertical edges, similar argument allows us to deduce that the intersection graphs of pairs of rows are $r$-partite with all vertex classes of size close to $N/r$. We reach a contradiction by showing that $r$-partite graphs with almost equal vertex classes cannot be very quasirandom. We are now ready to proceed with the proof.\\[12pt]
\boldSubsection{Proof}
Let $\chi$ be an arbitrary colouring of the grid with $r$ colours. For each $j \in [N]$, we denote by $H_j$ the complete graph on the vertex set $[N]$ along with an $r$-edge-colouring of the $j$\textsuperscript{th} row by $\chi$, i.e.
\[H_j = \Big\{(xy, \kappa) \in [N]^{(2)} \times [r] \colon \chi\Big((x,j)(y,j)\Big) = \kappa\Big\}.\] 
We misuse the notation slightly, and we define $H_{j_1} \cap H_{j_2}$ as
\[H_{j_1} \cap H_{j_2} = \Big\{xy \in [N]^{(2)} \colon (\exists \kappa \in [r]) (xy, \kappa) \in H_{j_1}\text{ and }(xy, \kappa) \in H_{j_2}\Big\},\]
thus the notation $H_{j_1} \cap H_{j_2}$ in our sense is exactly the projection of the usual $H_{j_1} \cap H_{j_2}$ under the mapping $(xy, \kappa) \mapsto xy$. Thus, as noted in~\cite{gridRamsey}, the condition that $\chi$ has no alternating rectangles is exactly the same as all graphs $H_{j_1} \cap H_{j_2}$ on the vertex set $[N]$ having chromatic number at most $r$. The following lemma shows that such graphs cannot be too quasirandom. For a graph $G$, we write $\hom(C_4, G)$ for the number of homomorphisms from $C_4$ to $G$, i.e. the number of $(x,y,z,w) \in V(G)^4$ such that $xy, yz, zw, wx \in E(G)$. This is very similar to counting copies of $C_4$ inside $G$, except that we allow repetitions of vertices, and the order matters, which is helpful when performing calculations.

\begin{lemma}\label{nonRandom}Let $G$ be a graph of order $N$. Suppose that $G$ is $k$-partite with vertex classes of size at most $(1 + \epsilon) \frac{N}{k}$ for some $\epsilon \in [0,1]$. Let $\delta = 2e(G)/N^2$ be the density of $G$. Then we have
\[\hom(C_4, G) \geq (1 - 4\epsilon)\Big(1 + \frac{1}{(k-1)^3}\Big)\delta^4 N^4.\]\end{lemma}

\begin{proof}Let $V$ be the set of vertices of the graph and let $V_1, \dots, V_k$ be the vertex classes. Let $\delta_{ij}$ be the density of the bipartite graph on vertex classes $V_i$ and $V_j$, i.e.\ $\delta_{i,j} = \frac{e(V_i, V_j)}{|V_i||V_j|}$. For vertices $x$ and $y$ we write $d_{x,y}$ for the \emph{codegree} of $x$ and $y$, which is the number of common neighbours. (If $x = y$, then $d_{x,y} = d_x$.) We have
\begin{equation}\label{firstIneq}\begin{split}\hom(C_4, G) = \sum_{x,y \in V} d_{x,y}^2 &= \sum_{i\in [k]} \sum_{x,y \in V_i} d_{x,y}^2 + \sum_{\substack{i,j \in [k] \\i \not= j}} \sum_{x \in V_i, y \in V_j} d_{x,y}^2\\
& \geq \sum_{i \in [k]} \frac{1}{|V_i|^2}\Big(\sum_{x,y \in V_i} d_{x,y}\Big)^2 + \sum_{i \not= j} \frac{1}{|V_i||V_j|}\Big(\sum_{x \in V_i, y \in V_j} d_{x,y}\Big)^2\\
& = \sum_{i \in [k]} \Big(\sum_{x,y \in V_i} \frac{d_{x,y}}{\sqrt{|V_i| \cdot |V_i|}}\Big)^2 + \sum_{\substack{i,j \in [k] \\i \not= j}} \Big(\sum_{x \in V_i, y \in V_j} \frac{d_{x,y}}{\sqrt{|V_i| \cdot |V_j|}}\Big)^2\\
& \geq \frac{1}{k}\Big(\sum_{i\in [k]} \sum_{x,y \in V_i} \frac{d_{x,y}}{\sqrt{|V_i| \cdot |V_i|}}\Big)^2 + \frac{1}{k(k-1)}\Big(\sum_{\substack{i,j \in [k] \\i \not= j}} \sum_{x \in V_i, y \in V_j} \frac{d_{x,y}}{\sqrt{|V_i| \cdot |V_j|}}\Big)^2\\
\end{split}\end{equation}
Let
\[T = \sum_{i,j \in [k]} \sum_{x \in V_i, y \in V_j} \frac{d_{x,y}}{\sqrt{|V_i| \cdot |V_j|}}\]
and
\[S = \sum_{i\in [k]} \sum_{x,y \in V_i} \frac{d_{x,y}}{\sqrt{|V_i| \cdot |V_i|}}.\] 
Write $d^i_z$ to be the number of edges between a vertex $z$ and class $V_i$. We thus have
\[\begin{split}S = \sum_{i\in [k]} \sum_{x,y \in V_i} \frac{d_{x,y}}{\sqrt{|V_i| \cdot |V_i|}} =& \sum_{\substack{i,j \in [k] \\i \not= j}} \sum_{z \in V_j} \Big(\frac{{d^i_z}}{\sqrt{|V_i|}}\Big)^2 =  \sum_{j \in [k]}\sum_{z \in V_j} \sum_{\substack{i \in [k] \\i \not= j}} \Big(\frac{{d^i_z}}{\sqrt{|V_i|}}\Big)^2\\
\geq& \frac{1}{k-1}\sum_{j \in [k]} \sum_{z \in V_j} \Big(\sum_{\substack{i \in [k]\\i \not= j}} \frac{{d^i_z}}{\sqrt{|V_i|}}\Big)^2\\
= &\frac{1}{k-1}\sum_{j \in [k]}\sum_{z \in V_j} \sum_{\substack{i, i' \in [k]\\i, i' \not= j}} \frac{d^i_z d^{i'}_z}{\sqrt{|V_i| |V_{i'}|}}\\
= &\frac{1}{k-1} \sum_{i, i' \in [k]}  \sum_{z \in V} \frac{d^i_z d^{i'}_z}{\sqrt{|V_i| |V_{i'}|}}\\
= &\frac{1}{k-1} \sum_{i, i' \in [k]}  \sum_{x \in V_i, y \in V_{i'}}\frac{d_{x,y}}{\sqrt{|V_i| |V_{i'}|}} = \frac{1}{k-1}T.\end{split}\]
Using the elementary fact that, for a fixed $T$, the quadratic function $x \mapsto x^2 - \frac{2}{k} Tx + \frac{1}{k}T^2$ is increasing on $[\frac{1}{k}T, \infty)$, the inequality~(\ref{firstIneq}) becomes (since $T \geq 0$ and $S \geq \frac{1}{k-1}T$)
\[\begin{split}\hom(C_4, G) &\geq \frac{1}{k} S^2 + \frac{1}{k(k-1)} (T-S)^2 = \frac{1}{k-1} S^2 - \frac{2}{k(k-1)} TS + \frac{1}{k(k-1)} T^2\\
&= \frac{1}{k-1} \bigg(S^2 - \frac{2}{k} TS + \frac{1}{k} T^2 \bigg) \geq \frac{1}{k-1} \bigg(\Big(\frac{T}{k-1}\Big)^2 - \frac{2}{k} T \Big(\frac{T}{k-1}\Big) + \frac{1}{k} T^2\bigg)\\
&= \frac{1}{k^2} \Big(1 + \frac{1}{(k-1)^3}\Big) T^2.
\end{split}\]
But 
\[\begin{split}T &= \sum_{i,i' \in [k]} \sum_{x \in V_i, y \in V_{i'}} \frac{d_{x,y}}{\sqrt{|V_i||V_{i'}|}} = \sum_{z \in V} \Big(\sum_{i, i' \in [k]} \frac{d^i_z d^{i'}_z}{\sqrt{|V_i||V_{i'}|}}\Big)\\
&=\sum_{j \in [k]} \sum_{z \in V_j} \Big(\sum_{\substack{i \in [k]\\i \not= j}} \frac{{d^i_z}}{\sqrt{|V_i|}}\Big)^2 \geq \sum_{j \in [k]} \frac{1}{|V_j|} \Big(\sum_{z \in V_j} \sum_{\substack{i \in [k]\\i \not= j}} \frac{{d^i_z}}{\sqrt{|V_i|}}\Big)^2 \\
&= \sum_{j \in [k]} \Big(\sum_{z \in V_j} \sum_{\substack{i \in [k]\\i \not= j}} \frac{{d^i_z}}{\sqrt{|V_i| \cdot |V_j|}}\Big)^2 \geq \frac{1}{k}\Big(\sum_{j \in [k]} \sum_{z \in V_j} \sum_{\substack{i \in [k]\\i \not= j}} \frac{{d^i_z}}{\sqrt{|V_i| \cdot |V_j|}}\Big)^2\\
&\geq \frac{1}{k}\frac{k^2}{(1+\epsilon)^2N^2}\Big(\sum_{j \in [k]} \sum_{z \in V_j} \sum_{\substack{i \in [k]\\i \not= j}} {d^i_z}\Big)^2 = \frac{k \delta^2 N^2}{(1+\epsilon)^2}.\end{split}\]
Thus, (using the elementary fact that $(1+\epsilon)^{-4} \geq (1-\epsilon)^4\geq 1-4\epsilon$ for $\epsilon \in [0,1]$) the number of homomorphisms of $C_4$ to $G$ is
\[\hom(C_4, G) \geq (1 - 4\epsilon)\Big(1 + \frac{1}{(k-1)^3}\Big)\delta^4 N^4,\]
as required.\end{proof}

For the next proposition, we need to generalize our notation slightly. For a set $A$, we write $K_A$ for the complete graph with the vertex set $A$. For example, the usual notation $K_N$ could be written as $K_{[N]}$ instead.

\begin{proposition}\label{dichotomy}Let $A, B \subset [N]$ with $|B| \geq r^5$. Suppose that $\chi \colon E(K_A \times K_B) \to [r]$ is an $r$-edge-colouring without alternating rectangles. Then, there is an edge-coloured $C_4$ $\Big((a_1a_2, \kappa_1),$ $(a_2a_3, \kappa_2),$ $(a_3a_4, \kappa_3),$ $(a_4a_1, \kappa_4)\Big)$ such that for at least $\Big(1 + \frac{1}{4}r^{-3}\Big)\frac{1}{r^4}|B|$ of $b \in B$ $b$\textsuperscript{th} row contains this pattern, or there is a coloured edge $(b_1b_2, \kappa)$ such that for at least $\Big(1 + \frac{1}{8}r^{-3} + O(r^{-4})\Big)\frac{1}{r}|A|$ of $a \in A$ $a$\textsuperscript{th} column contain this pattern.\end{proposition}

\noindent\textbf{Remark.} We could in principle make this proposition more explicit by writing $20r^{-4}$ instead of $O(r^{-4})$, when $r \geq 100$, for example (the choices of constants 20 and 100 are somewhat arbitrary), but we opt not to do so, to make the exposition clearer.

\begin{proof}Suppose the contrary. First of all, for each  $\Big((a_1a_2, \kappa_1), (a_2a_3, \kappa_2), (a_3a_4, \kappa_3), (a_4a_1, \kappa_4)\Big)$, we get at most $\Big(1 + \frac{1}{4r^3}\Big)\frac{1}{r^4}|B|$ of $b \in B$ that contain the given pattern in $b$\textsuperscript{th} row. Hence, (writing $a_5 = a_1$)
\[\begin{split}\sum_{\substack{b_1, b_2 \in B\\b_1 \not= b_2}} \hom(C_4, H_{b_1} \cap H_{b_2}) \leq & \sum_{a_1, a_2, a_3, a_4 \in A} \sum_{\kappa_1, \kappa_2, \kappa_3, \kappa_4 \in [r]} \Big|\Big\{b \in B \colon (\forall i \in [4])\chi((a_i, b)(a_{i+1},b)) = \kappa_i\Big\}\Big|^2\\
\leq &r^4 |A|^4 \Big(1 + \frac{1}{4r^3}\Big)^2\frac{1}{r^8}|B|^2 = \Big(1 + \frac{1}{4r^3}\Big)^2 \frac{1}{r^4}|A|^4|B|^2.\end{split}\]
For $b_1,b_2 \in B$, let $\delta_{b_1,b_2}$ stand for the density of $H_{b_1} \cap H_{b_2}$, i.e. $\delta_{b_1,b_2} = \frac{2e(H_{b_1} \cap H_{b_2})}{|A|^2}$. Notice that
\[\begin{split}\sum_{\substack{b_1, b_2 \in B\\b_1 \not= b_2}} \delta_{b_1,b_2}^4 |A|^4 = &16|A|^{-4} \sum_{\substack{b_1, b_2 \in B\\b_1 \not= b_2}} e(H_{b_1} \cap H_{b_2})^4 \\
\geq& 16|A|^{-4}|B|^{-6} \Big(\sum_{\substack{b_1, b_2 \in B\\b_1 \not= b_2}} e(H_{b_1} \cap H_{b_2})\Big)^4 \\
\geq& 16|A|^{-4}|B|^{-6} \Big(1 - r|B|^{-1}\Big)^4\Big(\sum_{b_1, b_2 \in B} e(H_{b_1} \cap H_{b_2})\Big)^4,\end{split}\] 
where the last inequality follows from easy bounds $\sum_{b \in B} e(H_b) \leq |B|\binom{|A|}{2}$ and $\sum_{b_1,b_2}e(H_{b_1} \cap H_{b_2}) \geq \frac{1}{r}|B|^2\binom{|A|}{2}$. Note that, since $r|B|^{-1} \leq 1$, we also have $\Big(1 - r|B|^{-1}\Big)^4 \geq 1 - 4r|B|^{-1}$. We proceed further
\[\begin{split}\sum_{\substack{b_1, b_2 \in B\\b_1 \not= b_2}} \delta_{b_1,b_2}^4 |A|^4\geq& \Big(1 - 4r|B|^{-1}\Big)|A|^{-4}|B|^{-6}\bigg(\sum_{a_1, a_2 \in A} \sum_{\kappa \in [r]} \Big(\sum_{b \in B} \bm{1}(\chi((a_1,b)(a_2,b)) = \kappa)\Big)^2\bigg)^4\\
\geq &\Big(1 - 4r|B|^{-1}\Big)r^{-4}|A|^{-12}|B|^{-6}\Big(\sum_{a_1, a_2 \in A} \sum_{\kappa \in [r]} \sum_{b \in B} \bm{1}(\chi((a_1,b)(a_2,b)) = \kappa)\Big)^8\\
= &\Big(1 - 4r|B|^{-1}\Big)r^{-4} |A|^4|B|^2.\end{split}\]
In particular,
\[\sum_{\substack{b_1, b_2 \in B\\b_1 \not= b_2\\\delta_{b_1,b_2}\geq \frac{1}{r^2}}} \delta_{b_1,b_2}^4 |A|^4 \geq \Big(1 - 4r|B|^{-1} - r^{-4}\Big)r^{-4} |A|^4|B|^2,\]
so we get
\[\sum_{\substack{b_1, b_2 \in B\\b_1 \not= b_2\\\delta_{b_1,b_2}\geq \frac{1}{r^2}}} \hom(C_4, H_{b_1} \cap H_{b_2}) \leq  \Big(1 + \frac{1}{4r^3}\Big)^2\Big(1 - 4r|B|^{-1} - r^{-4}\Big)^{-1} \sum_{\substack{b_1, b_2 \in B\\b_1 \not= b_2\\\delta_{b_1,b_2}\geq \frac{1}{r^2}}} \delta_{b_1,b_2}^4 |A|^4.\]
Since $|B| \geq r^5$ we deduce
\[\sum_{\substack{b_1, b_2 \in B\\b_1 \not= b_2\\\delta_{b_1,b_2}\geq \frac{1}{r^2}}} \hom(C_4, H_{b_1} \cap H_{b_2}) \leq \Big(1 + \frac{1}{2}r^{-3} + O(r^{-4})\Big)\sum_{\substack{b_1, b_2 \in B\\b_1 \not= b_2\\\delta_{b_1,b_2}\geq \frac{1}{r^2}}} \delta_{b_1,b_2}^4 |A|^4\]
We may thus find distinct $b_1, b_2 \in B$ such that $\delta_{b_1,b_2} \geq r^{-2}$ and $\hom(C_4, H_{b_1} \cap H_{b_2}) \leq \Big(1 + \frac{1}{2}r^{-3} + O(r^{-4})\Big) \delta_{b_1,b_2}^4 |A|^4$. Lemma~\ref{nonRandom} applies to give a vertex class of $H_{b_1} \cap H_{b_2}$ of size at least $\Big(1 + \frac{1}{8}r^{-3} + O(r^{-4})\Big) r^{-1}N$. But, this corresponds to $(b_1b_2, \kappa)$ appearing in that many columns for a suitable $\kappa \in [r]$, which is a contradiction to the initial assumptions.\end{proof}

We are now ready to prove the main result. 

\begin{proof}[Proof of Theorem~\ref{mainThm}]Suppose that $\chi \colon E(K_N \times K_N) \to [r]$ is an $r$-edge-colouring without an alterating rectangle. Our goal is to show that $N \leq \Big(1 - \frac{1}{128}r^{-2}(1  - o(1))\Big)r^{\binom{r+1}{2}}$. We iteratively apply Proposition~\ref{dichotomy}. For $i \in \Big[0, \Big\lfloor\frac{1}{8}r\Big\rfloor\Big]$, unless we have already deduced the desired bound on $N$, we show that there are a set of columns $A_i \subset [N]$, a set of rows $B_i \subset [N]$, a set of $4i$ coloured horizontal edges $E^{\text{hor}}_i \subset [N]^{(2)} \times [r]$, a set of $i$ coloured vertical edges $E^{\text{ver}}_i \subset [N]^{(2)} \times [r]$ and an integer $J_i \in [0,i]$ with the following properties:
\begin{itemize}
\item[$\bullet$] vertices of $E^{\text{hor}}_i$ are disjoint from $A_i$, vertices of $E^{\text{ver}}_i$ are disjoint from $B_i$,
\item[$\bullet$] for each $(a_1a_2, \kappa) \in E^{\text{hor}}_i$ and each $b\in B_i$, we have $\chi\Big((a_1,b)(a_2,b)\Big) = \kappa$,  for each $(b_1b_2, \kappa) \in E^{\text{ver}}_i$ and each $a\in A_i$, we have $\chi\Big((a,b_1)(a,b_2)\Big) = \kappa$,
\item[$\bullet$] $|E^{\text{hor}}_i| = 4(i  -J_i)$ and $|E^{\text{ver}}_i| = J_i$, and finally
\item[$\bullet$] $|A_i| \geq \Big(1 + \frac{1}{8}r^{-3} + O(r^{-4})\Big)^{J_i} r^{-J_i}N - 4(i-J_i)$, $|B_i| \geq \Big(1 + \frac{1}{4r^3}\Big)^{i-J_i}\frac{1}{r^{4(i-J_i)}}N - 2J_i$.
\end{itemize}
We prove this by induction on $i$, with the base case $i=0$ being trivial, where we may take $A_i = [N], B_i = [N], E^{\text{hor}}_i = \emptyset, E^{\text{ver}}_i = \emptyset, J_i = 0$. Suppose now that the claim holds for some $i \in \Big[0, \Big\lfloor\frac{1}{8}r\Big\rfloor\Big)$ and let $A_i, B_i, E^{\text{hor}}_i, E^{\text{ver}}_i, J_i$ the relevant sets and integer for $i$. Apply Proposition~\ref{dichotomy} to $\restr{\chi}{K_{A_i} \times K_{B_i}}$. We now discuss the two possible outcomes.\\
\indent Suppose that there is an edge-coloured $C_4$ $\Big((a_1a_2, \kappa_1),$ $(a_2a_3, \kappa_2),$ $(a_3a_4, \kappa_3),$ $(a_4a_1, \kappa_4)\Big)$, with $a_1,$ $a_2,$ $a_3,$ $a_4 \in A_i$ such that for at least $\Big(1 + \frac{1}{4r^3}\Big)\frac{1}{r^4}|B_i|$ of $b \in B_i$ $b$\textsuperscript{th} row contains this pattern. Then, we define 
\[\begin{split} A_{i+1} &= A_i \setminus \{a_1, a_2, a_3, a_4\},\\
E^{\text{hor}}_{i+1} &= E^{\text{hor}}_i \cup \Big\{(a_1a_2, \kappa_1), (a_2a_3, \kappa_2), (a_3a_4, \kappa_3), (a_4a_1, \kappa_4)\Big\},\\
B_{i+1} &= \Big\{b \in B_i \colon (\forall i \in [4])\hspace{2pt}\chi\Big((a_i, b)(a_{i+1}, b)\Big) = \kappa_i\Big\},\hspace{12pt}(\text{where }a_5 = a_1)\\
E^{\text{ver}}_{i+1} &= E^{\text{ver}}_i,\\
J_{i+1} &= J_i.\end{split}\]
Otherwise, there is a coloured edge $(b_1b_2, \kappa)$ such that for at least $\Big(1 + \frac{1}{8}r^{-3} + O(r^{-4})\Big)\frac{1}{r}|A_i|$ of $a \in A_i$ $a$\textsuperscript{th} column contain this pattern. In this case, we define
\[\begin{split} A_{i+1} &= \Big\{a \in A_i \colon \chi\Big((a, b_1)(a, b_2)\Big) = \kappa\Big\},\\
E^{\text{hor}}_{i+1} &= E^{\text{hor}}_i,\\
B_{i+1} &= B_i \setminus \{b_1, b_2\},\\
E^{\text{ver}}_{i+1} &= E^{\text{ver}}_i \cup \Big\{(b_1b_2, \kappa)\Big\},\\
J_{i+1} &= J_i + 1.\end{split}\]
It is easy to check that in both cases the given sets and integer satisfy the desired properties. Notice that this procedure can be completed for $i \leq \Big\lfloor\frac{1}{8}r\Big\rfloor$, as long the condition $|B_i| \geq r^5$ of Proposition~\ref{dichotomy} is satisfied. On the other hand, if that fails, then we certainly have the desired bound on $N$ for large enough $r$.\\
\indent To finish the proof, we look at the cases depending on $J_i$. Let $s = \Big\lfloor\frac{1}{8}r\Big\rfloor$.\\

\noindent\textbf{Case 1: }$J_s \geq s/2$\textbf{.} Let $V\subset [N]$ be any set of size $r+1$ that contains all vertices of edges in $E^{\text{ver}}_s$. List all edges in $V^{(2)} \setminus E^{\text{ver}}_s$ as $e_1, \dots, e_m$, where $m = \binom{r+1}{2} - J_s$. By pigeonhole principle, we may find colours $\kappa_1, \dots, \kappa_m \in [r]$ such that for at least $r^{-m}|A_s|$ of $a \in A_s$ we have colour $\kappa_i$ on edge $e_i$ in the column $a$. However, if we get at least two such $a_1, a_2 \in A_s$, then $\chi$ in these two columns coincides on $K_V$, which gives an alternating rectangle since $|V| = r  +1 $, as in Shelah's proof, resulting in a contradiction. Therefore,
\[\begin{split}1 \geq r^{-m}|A_s| \hspace{2pt}\geq\hspace{2pt}& r^{-\binom{r+1}{2} + J_s} \bigg(\Big(1 + \frac{1}{8}r^{-3} + O(r^{-4})\Big)^{J_s} r^{-J_s}N - 4r\bigg) \\
\geq\hspace{2pt}& \bigg(\Big(1 + \frac{1}{8}r^{-3} + O(r^{-4})\Big)^{\frac{r}{16}-1} r^{-\binom{r+1}{2}}N - 4r^{-\binom{r+1}{2} + r/8 + 2}\bigg),\end{split}\]
completing the proof in this case.\\

\noindent\textbf{Case 2: }$J_s < s/2$\textbf{.} This is very similar to the previous case, except that we now reverse the roles of rows and columns. Let $V\subset [N]$ be any set of size $r+1$ that contains all vertices of edges in $E^{\text{hor}}_s$. List all edges in $V^{(2)} \setminus E^{\text{hor}}_s$ as $e_1, \dots, e_m$, where $m = \binom{r+1}{2} - 4(s - J_s)$. By pigeonhole principle, we may find colours $\kappa_1, \dots, \kappa_m \in [r]$ such that for at least $r^{-m}|B_s|$ of $b \in B_s$ we have colour $\kappa_i$ on edge $e_i$ in the row $b$. However, if we get at least two such $b_1, b_2 \in B_s$, then $\chi$ in these two rows coincides on $K_V$, which once again gives an alternating rectangle, resulting in a contradiction. Therefore,
\[\begin{split}1 \geq r^{-m}|B_s| \hspace{2pt}\geq\hspace{2pt}& r^{-\binom{r+1}{2} + 4(s -J_s)} \bigg(\Big(1 + \frac{1}{4r^3}\Big)^{s-J_s}\frac{1}{r^{4(s-J_s)}}N - 2r\bigg)\\ \hspace{2pt}\geq\hspace{2pt} &\bigg(\Big(1 + \frac{1}{4r^3}\Big)^{\frac{r}{16} - 1}r^{-\binom{r+1}{2}}N - 2r^{4r + 1-\binom{r+1}{2}}\bigg),\end{split}\]
completing the proof.\end{proof}

\end{document}